\documentclass[a4paper, reqno]{amsart}

\usepackage{amsthm}
\usepackage{amsfonts}
\usepackage{amsmath}
\usepackage{mathrsfs}
\usepackage[colorlinks=true,citecolor=black,linkcolor=black]{hyperref}
\usepackage{graphicx}
\usepackage{epstopdf}
\usepackage{bbm}
\usepackage{mathtools}   
\usepackage{enumitem}

\theoremstyle{theorem}
\newtheorem*{main}{Main Theorem}
\newcommand{\citemain}{The Main Theorem}
\newtheorem{thm}{Theorem}[section]

\newtheorem{lem}[thm]{Lemma}

\theoremstyle{definition}

\DeclareMathAlphabet{\mathsc}{OT1}{cmr}{m}{sc}

\newcommand\cA{{\mathscr A}}
\newcommand\cB{{\mathscr B}}
\newcommand\cC{{\mathscr C}}

\newcommand\cI{{\mathcal I}}

\newcommand\cK{{\mathscr K}}
\newcommand\cL{{\mathscr L}}

\newcommand\cP{{\mathscr L}}

\newcommand\cS{{\mathcal S}}

\newcommand\bC{{\mathbb C}}

\newcommand\bN{{\mathbb N}}

\newcommand\bR{{\mathbb R}}

\newcommand{\osc}[2]{\operatorname{osc}\left[#1, #2   \right]}
\newcommand{\ball}[2]{{B}_{#1}( #2)}

\newcommand{\Oo}{\Omega} 
\newcommand{\flo}[1]{\Phi^{#1}} 
\newcommand{\Lspace}[2]{\mathbf{L^{#1}}(#2)} 
\newcommand{\indicator}[1]{\mathbf{1}_{#1}} 
\newcommand{\Id}{\mathbf{id}}   
\newcommand{\Cor}[2]{\operatorname{\xi}}   
   
\newcommand{\norm}[1]{\left\lVert{#1}\right\rVert}
\newcommand{\abs}[1]{\left\lvert{#1}\right\rvert}

\newcommand{\LL}[3]{{\lVert{#2}\rVert}_{\mathbf{L^{#1}}(#3)}}

\newcommand{\Aspace}{\mathfrak{B}_{\alpha}}   
\newcommand{\map}{f}

\newcommand\esssup{\operatorname{ess\, sup}\displaylimits}

\numberwithin{equation}{section}
\mathtoolsset{centercolon}   

\begin{document}
\author{Oliver Butterley}
\email{oliver.butterley@univie.ac.at}
\address{Faculty of Mathematics, Universit\"at Wien,
Nordbergstra{\ss}e 15,
1090 Wien, Austria}
\title{Area expanding $\cC^{1+\alpha}$ Suspension Semiflows}
\thanks{It is a pleasure to thank Carlangelo Liverani for many helpful discussions and comments. Research partially supported by the ERC Advanced Grant MALADY (246953).}
\date{  \today}
\bibliographystyle{abbrv}
\maketitle
\thispagestyle{empty}
\begin{abstract}
We study a large  class of suspension semiflows which contains the Lorenz semiflows. This is a class with low regularity (merely $\cC^{1+\alpha}$) and where the return map is discontinuous and the return time is unbounded.
We establish the functional analytic framework which is typically employed to study rates of mixing. The Laplace transform of the correlation function is shown to admit a meromorphic extension to a strip about the imaginary axis. 
As part of this argument we give a new result concerning the quasi-compactness of weighted transfer operators for piecewise $\cC^{1+\alpha}$ expanding interval maps.
\end{abstract}

\section{Introduction}
Some dynamical systems exhibit very good statistical properties in the sense of, for example, exponential decay of correlation and  the stability of the invariant measure under deterministic or random perturbations.  Such properties have been shown for many discrete-time dynamical systems and more recently for some flows. Very strong results now exist for smooth contact Anosov flows \cite{D,Li1,BL,tsujii2008qct,Tsujii:2011uq,Giulietti:fk}.
Good results also exist for suspension flows over uniformly-expanding Markov maps when the system is $\cC^{2}$ or smoother \cite{Pol99,baladi2005edc,avila2005emt}.
The above are all rather smooth and regular systems and arguably not  realistic or relevant in many physical systems. 
There are two important  examples which come to mind: dispersing billiards \cite{MR2229799} and the Lorenz flow \cite{Lorenz}. The fine statistical properties of both these systems remain, to some extent,  open problems. 
We therefore direct our interest to systems with rather low regularity.  Some recent progress includes the proof of exponential mixing for piecewise-cone-hyperbolic contact flows
\cite{Baladi:2011uq} and also for a class of three-dimensional singular flows \cite{Araujo:2011fk}.

This is our theme: To make some progress on the understanding of the fine results on statistical properties of systems with low regularity. 
The primary motivation for this study is the Lorenz flow mentioned above. This is a smooth three-dimensional singular hyperbolic flow. The work of Ara\'ujo and Varandas \cite{Araujo:2011fk} proved exponential decay of correlation for a class of volume-expanding flows with singularities, a class which is inspired by the Lorenz flow. However their method required the existence of a $\cC^{2}$ stable foliation. Unfortunately the stable foliation of the Lorenz flow is merely $\cC^{1+\alpha}$ and so there seems to be no hope of extending their strategy to the original problem. 
The problem of the stable foliations being merely $\cC^{1+\alpha}$  for Lorenz like flows has been partially tackled by Galatolo and Pacifico \cite{zezegala2010} followed by Ara\'ujo, Galatolo and Pacifico \cite{Araujo:2012fk} but results on decay of correlations are limited to the return map and not the flow. In this paper we make some progress in a complementary direction.
There exist two major strategies for approaching this problem: The first possible strategy is to construct an anisotropic Banach space in order to study the flow directly as was done for contact Anosov flows   \cite{Li1} and piecewise-cone-hyperbolic contact flows
\cite{Baladi:2011uq}. The second possible strategy is to study the \emph{Lorenz semiflow} (details given in Section~\ref{sec:lorenz}) which is given by quotienting along the stable manifolds. At this stage it is unclear how to construct the space required for the first possibility and we therefore consider the second. This however requires one to work with a system which is merely $\cC^{1+\alpha}$.

In this paper we focus on a particular class of  semiflows which are suspensions over expanding interval maps. This class includes the Lorenz semiflows. They have low regularity in the following four ways:
\begin{enumerate}
\item
The expansion of the return map may be unbounded. I.e. the derivative of the return map blows up close to certain points of discontinuity. This is an issue seen in both billiard systems and the Lorenz flow.
\item
The inverse of the derivative of the return map is merely H\"older continuous.
\item
The return time function is unbounded. This is a direct result of the zeros of the vector field associated to the flow. However in the case of certain suspension semiflows this has already been shown to not be a barrier to good statistical properties \cite{baladi2005edc}.
\item
The semiflow is merely area-expanding and not uniformly-expanding in the sense that it is not possible to define an invariant conefield which is uniformly transversal to the flow direction. This puts us in the category of singular hyperbolicity  \cite{morales1999shs, MR2123928}. 
\end{enumerate}
In order to study the class of flows considered in this paper, and other systems which are the object of current research, it is crucial to understand whether these above issues are real barriers to good statistical properties or merely technical difficulties. On this issue we succeed in making some progress in the present work showing that the listed issues are not real barriers to the statistical properties, at least in this setting.
For proving exponential decay of correlation for flows there is one established approach which involves studying the Laplace transform of the correlation function. We apply this strategy to our present setting and show that the Laplace transform of the correlation function admits a meromorphic extension into the left half plane. 
In Section~\ref{sec:results} we define precisely the class of semiflows we are interested in and state the results. In Section~\ref{sec:lorenz} we discuss Lorenz flows and demonstrate  the connection with the class of semiflows we consider. 
In Section~\ref{sec:gbv} we give a generalisation of the result of Keller on ``generalised bounded variation'' \cite{keller1985generalized} such that it is possible to apply to  our present application. This is a new result for the essential spectral radius of such transfer operators for these piecewise expanding interval maps and the section is independent of the others.
Section~\ref{sec:proof} contains the proof of the main result, reducing the problem to the study of certain weighted transfer operators and then using the results of Section~\ref{sec:gbv}.

\section{Results}\label{sec:results}

For our purposes we define a \emph{suspension semiflow} to be the triple $(\Oo, \map, \tau)$: Where
$\Oo$ is an open interval and 
 $\{\omega_{i}\}_{i\in \cI}$ is a finite or countable set of disjoint open sub-intervals 
which exhaust $\Oo$ modular a set of zero Lebesgue measure;
Where  $f \in \cC^{1}(\tilde\Oo, \Oo)$ (for convenience let $\tilde\Oo = \bigsqcup_{i\in \cI} \omega_{i}$)
 is a bijection when restricted to each $\omega_{i}$; And where $\tau\in\cC^{0}(\Oo,\bR_{+})$ such that $\int_{\Oo} \tau(x) \ dx <\infty$. In a moment we will add some stronger assumptions on the regularity of $\map$ and $\tau$.
We call $f$  the return map and $\tau$ the return time function.
Suppose that   $(\Oo, \map, \tau)$  is given. Let $\Oo_{\tau}:=\{(x,s) : x\in \tilde\Oo, 0\leq s<\tau(x)\}$  which we call the state space. For all $(x,s)\in \Oo_{\tau}$ and $t\in [0, \tau(x)-s]$ let
\begin{equation}\label{eq:defphi}
\flo{t}(x,s) := \begin{cases}
(x,s+t) & \text{if $t < \tau(x)-s$}\\
(\map (x),0) & \text{if $t = \tau(x)-s$}.
\end{cases}
\end{equation}
Note that $\flo{u+t}(x,s) = \flo{u} \circ \flo{t}(x,s)$ for all $u,t$ such that each term is defined. The flow is then defined for all $t\geq 0$ by assuming that this relationship continues to hold.

 Now we define the class of suspension semiflows which we will study. Firstly we require that the return map is expanding, i.e. that\footnote{In general it is sufficient to suppose that there exists $n\in \bN$ such that $\smash{\LL{\infty}{1/(\map^{n})'}{\Oo} <1}$. In which case one simply considers the $n$\textsuperscript{th} iterate of the suspension flow and proceeds as before, although care must be taken with assumption \eqref{eq:holderass}.} $\LL{\infty}{1/\map'}{\Oo} <1$.  We suppose that there exist some $\alpha\in (0,1)$ and $\sigma>0$ such that the following three conditions hold.
 Firstly we must have some, albeit weak, control on the regularity. We assume that\footnote{
 We say that some $\xi:\Oo\to \bC$ is  ``$\alpha$-H\"older   on $\Oo$'' if there exists $H_{\xi}<\infty$ such that $\abs{\xi(x)-\xi(y)} \leq H_{\xi} \abs{x-y}^{\alpha}$ for all $x,y\in\Oo$ with the understanding that this inequality  is trivially satisfied if $x\in \omega_{i}$, $y\in \omega_{i'}$, $i\neq i'$ since in this case $\abs{x-y}$ is not finite. Note that  $H_{\xi}$ does not depend on $i$.
 }  
 \begin{equation}\label{eq:holderass}
 x\mapsto \frac{e^{z \tau(x)}}{f'(x) }\quad \quad \text{is $\alpha$-H\"older  on $\tilde\Oo$ for each $\Re(z) \in [-\sigma,0]$}.
 \end{equation}
 Furthermore we must require sufficient expansion in proportion to the return time. We assume that
 \begin{equation}\label{eq:expandingass}
 \sup_{i\in \cI} \left({\LL{\infty}{\tfrac{1}{\map'}}{\omega_{i}}}\right)^{\alpha}e^{{\sigma} \LL{\infty}{\tau}{\omega_{i}} } <1
 \end{equation}
 Finally, to deal with the possibility of a countable and not finite number of disconnected components of $\Oo$, we assume that
 \begin{equation}\label{eq:sumass}
  \sum_{i\in \cI} \LL{\infty}{\tfrac{1}{\map'}}{\omega_{i}} e^{\sigma \LL{\infty}{\tau}{\omega_{i}} } <\infty.
 \end{equation}
 Note that we never require any lower bound on $\tau$.
Let $\nu$ denote some $\map$-invariant probability measure which is absolutely continuous with respect to Lebesgue on $\Oo$.  The existence of such a probability measure is already known but is also implied by the results of Section~\ref{sec:gbv}.
For simplicity we assume that this absolutely continuous invariant probability measure is unique.
It holds that $\mu :=\nu \otimes \operatorname{Leb} / \nu(\tau)$ is a $\flo{t}$-invariant probability measure which is absolutely continuous with respect to Lebesgue on $\Oo_{\tau}$.
Given $u, v:\Oo_{\tau } \to \bC$ which are $\alpha$-H\"older  we define for all $t\geq 0$ the correlation 
\[
\Cor{u}{v}(t) := \mu( u \cdot v\circ \flo{t} ) - \mu(u) \cdot \mu(v).
\]
\begin{main}\label{thm:main}
Suppose the suspension semiflow is as described above, in particular satisfying the assumptions \eqref{eq:holderass}, \eqref{eq:expandingass}, and \eqref{eq:sumass}.
Then the Laplace transform of the correlation $\widehat{\Cor{u}{v}}(z) := \int_{0}^{\infty} e^{-zt} \Cor{u}{v}(t) \ dt$ admits a meromorphic extension to the set $\{z\in \bC: \Re(z) \geq -\sigma\}$.
\end{main}
\noindent
The proof of this theorem is given in Section~\ref{sec:proof} and is based on the results of  Section~\ref{sec:gbv}. The argument involves the usual method of ``twisted transfer operators'' but for this setting we require a generalisation of Keller's previous work \cite{keller1985generalized} on $\cC^{1+\alpha} $ expanding interval maps which is the content of  Section~\ref{sec:gbv}.

Let us recall in detail some closely related results which were mentioned in the introduction. Baladi and Vall\'ee \cite{baladi2005edc} (argument later extended to higher dimensions \cite{avila2005emt} by Avila, Yoccoz and Gou\"ezel) studied suspension semiflows which had  return maps which were Markov and also $\cC^{2}$. They allowed the return time to be unbounded but only in a mild way as they required $\tau'/\map'$ to be bounded. As part of the study of Lorenz-like flows Ara\'ujo and Varandas \cite{Araujo:2011fk} studied suspension semiflows very similar to the present setting but had to additionally require that the return map was $\cC^{2}$ rather than our weaker assumption of $\cC^{1+\alpha}$. We therefore see that our setting is more general and sufficiently general to be used for the study of Lorenz flows (see Section~\ref{sec:lorenz}). However in each of the above mentioned cases exponential decay of correlation is proven, a significantly stronger result that is proven in this present work. To obtain results on exponential decay of correlation in would be additionally necessary to have an oscillatory cancelation argument as pioneered by Dolgopyat \cite{D}. But since we already have the Lasota-Yorke inequality which we prove in Section~\ref{sec:gbv}, it is suffices to prove this stronger estimate in the $\mathbf{L^{1}}$ norm. All indication therefore suggests that, although far from trivial, there is hope of proving such an estimate.

\section{Lorenz Semiflows}\label{sec:lorenz}
Introduced in 1963 as a simple model for weather, the Lorenz flow \cite{Lorenz}   is a smooth three dimensional flow which, from numerical simulation, appeared to exhibit a robust chaotic attractor.
In the late 1970s 
Afra\u{\i}movi\v{c},
  Bykov and Silnikov  \cite{MR0462175}
 and Guckenheimer and Williams 
 \cite{MR556582,MR556583}
  introduced a geometric model of the Lorenz flow and some years later, in 2002, 
Tucker  \cite{T02} showed that the geometric Lorenz flow really was a representative model for the original Lorenz flow and hence showed that the Lorenz attractor really did exist.

 This flow has long proved elusive to thorough study. It is not uniformly hyperbolic.
The class of \emph{singular hyperbolic flows} was introduced and studied in the late 1990s by   Morales, Pacifico, Pujals \cite{morales1999shs, MR2123928,MR1649489}.
This class of flows contains the uniform hyperbolic flows and also contains the Lorenz attractor. 
Whereas the uniformly hyperbolic flows are  the flows which are structurally stable as shown by Hayashi   \cite{MR1432037,MR1715329},   the singular hyperbolic flows are  the flows which are stably transitive.
It is known that singular hyperbolic flows are chaotic in that they are expansive and admit an SRB measure   \cite{APPV}. Some further results are known limited to the particular case of the Lorenz attractor. It is know to be mixing \cite{LuzzMelPac2005} and that the Central Limit Theorem  and Invariance Principle hold  \cite{holland2007clt}. 
As mentioned earlier, a class of Lorenz-like flows has been shown to mix exponentially 
\cite{Araujo:2011fk} although this result is limited to such flows which have $\cC^{2}$ stable foliations, a property which cannot be expected to hold in general or for the original Lorenz flow.
 
 To show that the Lorenz flow reduces to a suspension semiflow of the class introduced in Section~\ref{sec:results} it is necessary to collect together some known facts. This is a procedure similar to the one described in  \cite{LuzzMelPac2005}.
Firstly it is known that stable manifolds exist for the Lorenz attractor and they are $\cC^{1+\gamma} $ for some $\gamma\in (0,1)$. 
Quotienting along these stable manifolds  the three dimensional  Lorenz flow may be reduced to a $\cC^{1+\gamma} $ area expanding semiflow on a two dimensional branched manifold where the equilibrium point is (or indeed, all the equilibrium points are) after a $\cC^{1+\gamma} $ change of coordinates to linearise close to the singularity, of the form
\[
\phi_{t}: (x,y) \mapsto (xe^{\lambda t},y e^{-\beta \lambda t}),
\]
where $\lambda>0$ and $\beta \in (0,1)$.
The equilibrium point of the Lorenz flow is of saddle type with two negative eigenvalues and one positive eigenvalue. We write the eigenvalues as $\zeta_{ss} < \zeta_{s} <0 < \zeta_{u}$.  The dominated splitting implies that $\zeta_{u} > \abs{\zeta_{s}}$. This is the origin of the above two coefficients. We have  $\lambda= \zeta_{u}$ and $\beta = \abs{\zeta_{s}} / \zeta_{u}$.

We now choose a suitable Poincar\'e section (made of, perhaps, many disconnected components) and thereby reduce the semiflow to a suspension described by a return map and return time function. Away from the regions containing  an equilibrium point this is simple and yields a $\cC^{1+\gamma}$ return map and return time function where the return time function is bounded. Closer to the singularities the construction is slightly more delicate.
For $x\in (0,1)$ let
\[
\tau(x):= -\tfrac{1}{\lambda} \ln (x), \quad \quad \map(x):= x^{\beta}.
\]
These definitions have the useful consequence that  $\phi_{\tau(x)}:(x,1)\mapsto (1,\map(x))$ for all $x\in (0,1)$ and that $\phi_{\tau(x)}:\{(x,1):x\in(0,1)\} \to \{(1,y):y\in(0,1)\}$ is a bijection. It is convenient to further subdivide these components of the Poincar\'e section and so for each $i\in \bN$ let $\omega_{i}:= (e^{-(i+1)},e^{-i})$. 
We must verify that the conditions \eqref{eq:holderass}, \eqref{eq:expandingass}, and \eqref{eq:sumass} are satisfied for this suspension semiflow. 
We choose $\alpha:=\min \{ \gamma, (1-\beta)/(2-\beta)\}$ and $\sigma>0$ such that 
\begin{equation}\label{eq:green}
\sigma< \alpha\lambda(1-\beta),
\end{equation}
 the larger the better. Note  this implies that $\alpha \in (0,\frac{1}{2})$ and that
\begin{equation}\label{eq:blue}
\sigma\leq \lambda(1-\beta-\alpha).
\end{equation}
Let $\Re(z)\in [-\sigma,0]$. First note that
\[
\frac{e^{-z\tau(x)}}{\map'(x)} = x^{\frac{z}{\lambda} + 1 -\beta}
\]
and that $\Re(z/\lambda + 1 -\beta)\geq -\sigma/\lambda + 1 -\beta \geq \alpha$ by \eqref{eq:blue}.
Note that $y^{\zeta} - x^{\zeta} = \zeta \int_{x}^{y} s^{\zeta-1} \ ds$ for all $x,y\in \bR$ and so for   $y\geq x$ we have
\[
{\smash{y^{\zeta} - x^{\zeta}}}
\leq \abs{\zeta} \int_{x}^{y} s^{\Re(\zeta)-1} \ ds
=  \frac{\abs{\zeta}}{ \Re(\zeta)} ({ \smash{y^{\Re(\zeta)}- x^{\Re(\zeta)}}  }).
\] 
Consequently $x\mapsto {e^{z\tau(x)}}/{\map'(x)} $ is $\alpha$-H\"older on $(0,1)$. We must now show that the other two estimates hold.
We have the simple estimates
\[
\LL{\infty}{\tfrac{1}{\map'}}{\omega_{i}} = e^{-i(1-\beta)},
\quad \quad
\LL{\infty}{\tau}{\omega_{i}} =\tfrac{1}{\lambda} (i+1).
\]
This means that
\[
\LL{\infty}{\tfrac{1}{\map'}}{\omega_{i}}^{\alpha}e^{{\sigma} \LL{\infty}{\tau}{\omega_{i}} } 
=e^{ -i[\alpha(1-\beta) - \sigma/ \lambda] } e^{\sigma/ \lambda}
\]
and by \eqref{eq:green} we know that $\alpha(1-\beta) - \sigma/ \lambda>0$.
This means that it is simple to arrange that \eqref{eq:expandingass} and \eqref{eq:sumass} are satisfied. Unfortunately we are not quite done since we have not shown that $\LL{\infty}{1/\map'}{\Oo}<1$. We do know however that there exists $n\in \bN$ such that  $\LL{\infty}{1/(\map^{n})'}{\Oo}<1$. Consequently we instead consider the $n$\textsuperscript{th} iterate suspension semiflow. Conditions \eqref{eq:expandingass} and \eqref{eq:sumass} are still satisfied by the iterate. However care must be taken by the H\"older continuity assumption \eqref{eq:holderass}. It may happen that this is now only satisfied for some $\tilde\alpha\in (0,\alpha)$. Consequently this smaller value of $\alpha$ must be used from the start of the construction. Actually the return map is bounded in $\cC^{1}$ norm away from a neighbourhood of the equilibrium point and therefore with some knowledge of returns to the region of the equilibrium point a more optimal $\alpha$ could be choosen. 
The above estimates mean that the results of the previous section apply to the Lorenz semiflows.

We make a few more comments about this particular suspension semiflow.
This suspension semiflow presents the difficulty that the return time is not bounded but moreover $\tau' / f'$ is not bounded. (Such a condition is crucially required in  
 \cite{baladi2005edc, avila2005emt}.) This means that the flow is not uniformly expanding in the sense of there existing an invariant conefield, uniformly bounded away from the flow direction, inside of which there is uniform expansion. 
Lorenz semiflows as discussed above are our main application although we study a more general class of suspension semiflows.

\section{Generalised Bounded Variation}\label{sec:gbv}
We must consider the weighted transfer operators associated to expanding maps of the interval which have countable discontinuities and for which the inverse of the derivative and the weighting are  merely H\"older continuous. This means that we cannot study the transfer operator acting on any relatively standard spaces. One possibility is the generalised bounded variation introduced by Keller \cite{keller1985generalized} and used for expanding interval maps. However he does not consider the case when there are countable discontinuities and also does not consider the case of general weights.
Saussol \cite{MR1759406} used the same spaces for multi-dimensional expanding maps and showed that countable discontinuities are allowable but again did not study a general classes of weights and furthermore required the derivative of the map to be bounded. These are the spaces we will use in this section. Although not proven in the references with delicate estimates these spaces, as we will prove in this section, are useful for our application. Other possibly useful Banach spaces are available \cite{MR2784627, oli1201, 2012arXiv1207.3982L} but each suffers from some limitation which prevents the use in this present setting without imposing undesirable further conditions on the semiflow we wish to study. One particular problem is that we cannot guarantee that the weighting is bounded (see Section~\ref{sec:proof}) and consequently we cannot guarantee that the weighted transfer operator is bounded on $\mathbf{L^{1}}$.  Keller's Banach space of   \emph{generalised bounded variation}  \cite{keller1985generalized} is contained within $\mathbf{L^{\infty}}$, a distinct difference to the available alternatives  \cite{ oli1201, 2012arXiv1207.3982L}. This suggests the possibility that the transfer operator is bounded on this space even when not bounded on $\mathbf{L^{1}}$. In the remainder of this section we see that this speculation is shown to be correct.

\subsection{The Banach Space}
The following definitions are identical to  \cite{keller1985generalized} with minor changes of notation.
For any interval $\cS$  and $h:\cS \to \bC$ let
\[
\osc{h}{\cS} := \esssup\left\{\abs{h(x_{1})-h(x_{2})}: x_{1},x_{2}\in \cS   \right\}
\]
where the essential supremum is taken with respect to Lebesgue measure on $\cS^{2}$. Let $\ball{\epsilon}{x} := \{y\in \bR: \abs{x-y}\leq \epsilon\}$.
If $\alpha \in (0,1)$ and $\Oo$ is some finite or countable union of open intervals let
\begin{equation}\label{eq:defnorm}
\abs{h}_{\Aspace}:= \sup_{\epsilon \in (0,\epsilon_{0})} \epsilon^{-\alpha} \int_{\Oo} \osc{h}{\ball{\epsilon}{x}\cap \Oo} \ dx,
\end{equation}
where $\epsilon_{0}>0$ is some fixed parameter. 
Hence let
\[
\mathfrak{B}_{\alpha}:= \left\{ h\in \Lspace{1}{\Oo} : \abs{h}_{\Aspace}<\infty\right\}.
\]
The seminorm defined above will depend on $\epsilon_{0}>0$ although the sets $\mathfrak{B}_{\alpha}$ do not. It is known  \cite[Theorem 1.13]{keller1985generalized} that this set is a Banach space when equipped with the norm
\[
\norm{h}_{\Aspace}:= \abs{h}_{\Aspace} + \LL{1}{h}{\Oo},
\]
that $\Aspace \subset \Lspace{\infty}{\Oo}$, and that the embedding
\begin{equation}\label{eq:compact}
\Aspace \hookrightarrow \Lspace{1}{\Oo}
\quad \quad \text{is compact}.
\end{equation}
\subsection{Piecewise Expanding Transformations}
As before,
we suppose that
$\Oo$ is an open interval and 
 $\{\omega_{i}\}_{i\in \cI}$ is a finite or countable set of disjoint open sub-intervals 
which exhaust $\Oo$ modular a set of zero Lebesgue measure
(for convenience let $\tilde\Oo = \bigsqcup_{i\in \cI} \omega_{i}$)
and that we are given  
\[
\map \in \cC^{1}(\tilde\Oo, \Oo)
\]
 which is  bijective when restricted to each $\omega_{i}$.
We further suppose that we are given $\xi:\Oo\to \bC$ 
which we call the \emph{weighting}.
We require that
\begin{equation}
 \label{eq:expanding}
 \LL{\infty}{1/ \map'}{\Oo} \in (0,1),
\end{equation}
furthermore that
\begin{equation}
 \label{eq:summing}
\sum_{i\in \cI}\LL{\infty}{1/ \map'}{\omega_{i}} \LL{\infty}{\xi}{\omega_{i}} <\infty,
\end{equation}
and finally that $\frac{\xi}{\map'}:\Oo \to \bC$ is $\alpha$-H\"older. I.e. there exists $H_{\xi}<\infty$ and $\alpha \in (0,1)$ such that 
\begin{equation}
 \label{eq:holder}
\abs{\tfrac{\xi}{\map'}(x) -\tfrac{\xi}{\map'}(y)}
  \leq H_{\xi} \abs{x-y}^\alpha \quad\quad \text{for all  $x,y\in \omega_i$ for each $i\in \cI$}. 
\end{equation}
For convenience let $\map_i:\omega_i \to \Oo$ denote the restriction of $\map$ to $\omega_i$.
As usual the weighted transfer operator  is given, for each $h:\Oo\to \bC$,  by\footnote{For any set, e.g. $A$, we let $ \indicator{A}$ denote the indicator function of that set.}
\begin{equation}\label{eq:defL}
\cL_{\xi}h(x):= \sum_{i\in\cI} \left(\frac{\xi \cdot h}{\map'}\right)\circ \map_{i}^{-1}(x)\cdot \indicator{\map\omega_{i}}(x).
\end{equation}
By \eqref{eq:summing} we know that $\cL_{\xi} :\Lspace{\infty}{\Oo} \to \Lspace{\infty}{\Oo} $ is well defined even thought, since we do not require $\LL{\infty}{\xi}{\Oo}<\infty$, we cannot guarantee that the operator is well defined on $\Lspace{1}{\Oo}$.

The purpose of this section is to prove the following new result which is a generalisation of the work of Keller \cite{keller1985generalized} to the case of countable discontinuities and unbounded weightings.
\begin{thm}\label{thm:qc}
Suppose the transformation $f:\Oo\to \Oo$ and the weighting $\xi:\Oo\to\bC$ are as above and satisfy \eqref{eq:expanding},   \eqref{eq:summing} and \eqref{eq:holder}. Then $\cL_{\xi}:\Aspace \to \Aspace$ is a bounded operator 
with essential spectral radius not greater than  
\[
\lambda:= \sup_{i\in \cI} \LL{\infty}{1/\map'}{\omega_{i}}^{\alpha} \LL{\infty}{\xi }{\omega_{i}}.
\]
\end{thm}
\noindent
By a standard argument (see for example \cite[p.1281]{Li1})  the essential spectral radius estimate of the above theorem follows from the  compact embedding \eqref{eq:compact} and the  Lasota-Yorke type estimate contained in the following theorem. 
In the case where $\LL{\infty}{\xi}{\Oo}<\infty$ an elementary estimate shows that $\LL{1}{\cL_{\xi}}{\Oo} \leq \LL{\infty}{\xi}{\Oo}$ and so, once the essential spectral radius estimate has been shown, this implies that  the spectral radius is not greater than $\LL{\infty}{\xi}{\Oo}$.
\begin{thm}\label{thm:LY}
Suppose that $\map$ and $\xi$ are as per the assumptions of Theorem~\ref{thm:qc}. Then for all $\delta>0$ there exists $C_{ \delta}<\infty$ such that
\[
\norm{\cL_{\xi}h}_{\Aspace}\leq (2+\delta)\lambda \norm{h}_{\Aspace} + C_{ \delta}\LL{1}{h}{\Oo} 
\quad \quad
\text{for all $h\in \Aspace$}.
\]
\end{thm}
\noindent
The remainder of this section is devoted to the proof of the above proposition.
This estimate is an extension of the result of Keller~\cite{keller1985generalized} to our setting. 
The proof follows a similar argument to Keller's original with various additional  complications, in particular because of the weighting $\xi$ and the possibility that $\cI$ is merely countable. As such we are forced to redo the proof but when possible we refer to the relevant theorems and lemmas which we can reuse.

\subsection{Proof of Theorem~\ref{thm:LY}}
We may assume that $\delta \leq 1$.
First $\epsilon_0>0$ must be carefully chosen and it is convenient to divide the index set as $\cI = \cI_1 \cup \cI_2$.
By \eqref{eq:summing} we may choose a finite set $\cI_1\subset\cI$ such that
\begin{equation}
 \label{eq:remainder}
  \sum_{i\in\cI_{2}} \LL{\infty}{1/\map'}{\omega_i} \LL{\infty}{\xi}{\omega_{i}} \leq \frac{\lambda \delta}{16}, 
\end{equation}
where $\cI_2:= \cI\setminus\cI_1$.
Let $\Gamma := 32 \delta^{-1}+2$.
Choosing $\epsilon_0$ sufficiently small we ensure that
\begin{equation}
 \label{eq:bigpieces}
  \abs{\map\omega_i}\geq \epsilon_0\Gamma \quad\quad\quad \text{for all $i\in \cI_1$}
\end{equation}
and that
\begin{equation}
 \label{eq:okey}
 \epsilon_0^\alpha \leq \frac{\delta  \lambda}{ 8  (8+ {\delta}) H_\xi \Gamma }.
\end{equation}
(The reason for this particular choice will subsequently become clear  \eqref{eq:est2}.)
If $ \abs{\map\omega_i}> 2 \epsilon_0\Gamma $ for some $i\in \cI_1$ we chop $\omega_i$ into pieces such that $\epsilon_0\Gamma  \leq \abs{\map\omega_j}\leq 2 \epsilon_0\Gamma $ for all the resulting pieces. 
If $ \abs{\map\omega_i}> 2 \epsilon_0\Gamma $ for some $i\in \cI_2$ we chop $\omega_i$ into pieces as before but in this case we move the resulting pieces into the set $\cI_1$. This means that the estimate \eqref{eq:remainder} remains unaltered. Note that $\cI_1$ may no longer be a finite set.
To conclude we have arranged so that \eqref{eq:remainder}, \eqref{eq:bigpieces}, and \eqref{eq:okey} hold and furthermore that
\begin{equation}
 \label{eq:smallpieces}
  \abs{\map\omega_i}\leq 2\epsilon_0\Gamma \quad\quad\quad \text{for all $i\in \cI$}.
\end{equation}

Fix $h\in \Aspace$. 
We start by noting that by the definition \eqref{eq:defnorm} of the seminorm and the definition \eqref{eq:defL} of the transfer operator
\begin{equation}\begin{split}\label{eq:wien}
\abs{\cL_{\xi}h}_{\Aspace} 
&= 
\sup_{\epsilon \in (0,\epsilon_{0})} \epsilon^{-\alpha}\int_{\Oo}\osc{\cL_{\xi}h}{ \ball{\epsilon}{x} \cap \Oo} \ dx\\
&\leq
 \sup_{\epsilon \in (0,\epsilon_{0})}  \sum_{i}   \epsilon^{-\alpha}\int_{\Oo} \osc{ \left(\tfrac{\xi \cdot h}{\map'}\right)\circ \map_{i}^{-1}\cdot \indicator{\map\omega_{i}} }{ \ball{\epsilon}{x} \cap \Oo} \ dx.
\end{split}\end{equation}
To proceed we take advantage of several estimates which have already been proved elsewhere.
Firstly by {\cite[Theorem 2.1]{keller1985generalized}}  for each $i\in \cI_1$, since $\abs{\map\omega_i}\geq (32\delta^{-1} +2)\epsilon_0$ by \eqref{eq:bigpieces} and  $\abs{\map\omega_i}\geq 4 \epsilon_0$, we have that
\begin{equation}\begin{split}
\label{eq:kel1}
&\int_{\Oo}\osc{ \left(\tfrac{\xi \cdot h}{\map'}\right)\circ \map_{i}^{-1}\cdot \indicator{\map\omega_{i}} }{ \ball{\epsilon}{x} \cap \Oo} \ dx\\
& \quad  \quad  \quad \leq (2+\tfrac{\delta}{4}) \int_{\map \omega_i} \osc{ \left(\tfrac{\xi \cdot h}{\map'}\right)\circ \map_{i}^{-1}}{ \ball{\epsilon}{x} \cap \map\omega_{i}} \ dx\\
& \quad  \quad  \quad + \frac{\epsilon}{\epsilon_{0}} \int_{\map\omega_i}  \abs{\tfrac{\xi \cdot h}{\map'}}\circ \map_{i}^{-1}(x) \ dx.
\end{split}\end{equation}
For  $i\in \cI_2$ (where  $\abs{\map\omega_i}$ may be small) we use the following, more basic estimate.
By \cite[Proposition 3.2 (ii)]{MR1759406}
for each $i$
\begin{equation*}\begin{split}
\osc{ \left(\tfrac{\xi \cdot h}{\map'}\right)\circ \map_{i}^{-1}\cdot \indicator{\map\omega_{i}} }{ \ball{\epsilon}{x} \cap \Oo} 
& \leq  \osc{ \left(\tfrac{\xi \cdot h}{\map'}\right)\circ \map_{i}^{-1}}{ \ball{\epsilon}{x} \cap \map\omega_{i}} \cdot \indicator{\map\omega_i}\\
& + 2  \LL{\infty}{\tfrac{\xi \cdot h}{\map'} }{\omega_i}   \indicator{F_{i,\epsilon}}(x) 
\end{split}\end{equation*}
where $F_{i,\epsilon}$ denotes the set of all points $x\in \Oo$ which are within a distance of $\epsilon$ of the end points of the interval $\map \omega_i$. Since 
$\abs{\smash{\int_{\map\omega_i} \indicator{F_{i,\epsilon}}(x) \ dx}} \leq 2 \epsilon$ the above implies  that
\begin{equation}\begin{split}
\label{eq:sau2}
&\int_{\Oo}\osc{ \left(\tfrac{\xi \cdot h}{\map'}\right)\circ \map_{i}^{-1}\cdot \indicator{\map\omega_{i}} }{ \ball{\epsilon}{x} \cap \Oo} \ dx\\
& \quad  \quad  \quad \leq  \int_{\map \omega_i} \osc{ \left(\tfrac{\xi \cdot h}{\map'}\right)\circ \map_{i}^{-1}}{ \ball{\epsilon}{x} \cap \map\omega_{i}} \ dx\\
& \quad  \quad  \quad + 4 \epsilon \LL{\infty}{\xi }{\omega_{i}} \LL{\infty}{h }{\Oo} \LL{\infty}{1/\map' }{\omega_i}.
\end{split}\end{equation}
Note that the integral term in the middle line of the above equation is identical to the integral term of the middle line of \eqref{eq:kel1}.
We also require the following basic estimate for the $\osc{\cdot}{\cdot}$ of a product.
\begin{lem}\label{lem:product}
Suppose $\cS\subset\Oo$ is an interval, $g_1:\cS \to \bC$,  $g_1:\cS \to \bC$ and $y\in \cS$. Then
\[
\osc{g_{1} \cdot g_{2}}{\cS} \leq \abs{g_{1}(y)} \cdot \osc{g_{2}}{\cS} + 2 \LL{\infty}{g_{2}}{\cS}  \cdot \osc{g_{1}}{\cS}.
\]
\end{lem}
\begin{proof}
Suppose $x_{1},x_{2}, y\in \cS$. It suffices to observe that
\[
\begin{split}
(g_{1} \cdot g_{2})(x_{1}) - (g_{1} \cdot g_{2})(x_{2}) 
&=  g_{1}(y)\left(g_{2}(x_{1})-g_{2}(x_{2})\right)\\
& + g_{2}(x_{1})   \left(g_{1}(x_{1})-g_{1}(y) \right)
 + g_{2}(x_{2})   \left(g_{1}(y)-g_{1}(x_{2}) \right).\qedhere
\end{split}
\]
\end{proof}
This means in particular that (this is the term which appears in the middle lines of \eqref{eq:kel1} and \eqref{eq:sau2})
\begin{equation}
 \label{eq:product}\begin{split}
 &\int_{\map \omega_i} \osc{ \left(\tfrac{\xi \cdot h}{\map'}\right)\circ \map_{i}^{-1}}{ \ball{\epsilon}{x} \cap \map\omega_{i}} \ dx\\
&\quad \quad \quad \quad \leq 
   \int_{\map \omega_i}  \abs{\tfrac{\xi}{\map'}}\circ\map_{i}^{-1}(x) \cdot \osc{h \circ \map_{i}^{-1}}{ \ball{\epsilon}{x} \cap \map\omega_{i}} \ dx\\
&\quad \quad \quad \quad + 2
\LL{\infty}{h}{\omega_i} \int_{\map \omega_i} \osc{ \tfrac{\xi }{\map'}\circ \map_{i}^{-1}}{ \ball{\epsilon}{x} \cap \map\omega_{i}} \ dx.
    \end{split}
\end{equation}
Recalling \eqref{eq:wien} and applying the estimates of \eqref{eq:kel1}, \eqref{eq:sau2} and \eqref{eq:product} we have
\begin{equation}\label{eq:infour}
 \abs{\cL_{\xi}h}_{\Aspace} 
\leq 
\sup_{\epsilon \in (0,\epsilon_{0})}
\left(A_{1,\xi,h}(\epsilon) +A_{2,\xi,h}(\epsilon) +A_{3,\xi,h}(\epsilon) +A_{4,\xi,h}(\epsilon)\right),
\end{equation}
where we have definded for convenience
\[\begin{split}
 A_{1,\xi,h}(\epsilon) 
 &:= \epsilon^{-\alpha} (2+\tfrac{\delta}{4}) \sum_{i\in\cI} \int_{\map \omega_i}\abs{\tfrac{\xi }{\map'}}\circ \map_{i}^{-1}(x)
 \cdot \osc{ h\circ \map_{i}^{-1}}{ \ball{\epsilon}{x} \cap \map\omega_{i}} \ dx   \\
 A_{2,\xi,h}(\epsilon) 
 &:=  2 \epsilon^{-\alpha} (2+\tfrac{\delta}{4}) \sum_{i\in\cI} \LL{\infty}{h}{\omega_i} \int_{\map \omega_i} \osc{ \tfrac{\xi }{\map'}\circ \map_{i}^{-1}}{ \ball{\epsilon}{x} \cap \map\omega_{i}} \ dx\\ 
 A_{3,\xi,h}(\epsilon) &:= 
 4 \epsilon^{1-\alpha} \LL{\infty}{h }{\Oo} \sum_{i\in\cI_2}  \LL{\infty}{1/\map' }{\omega_i} \LL{\infty}{\xi }{\omega_{i}}\\
 A_{4,\xi,h}(\epsilon) 
  &:=\frac{ \epsilon^{1-\alpha}}{\epsilon_0}\sum_{i\in\cI_1}
 \int_{\map \omega_i}\abs{\tfrac{\xi \cdot h}{\map'}}\circ \map_{i}^{-1}(x)\ dx.
\end{split}\]
The remainder of the proof involves independently estimating each of these four terms.

We start by estimating $ A_{1,\xi,h}(\epsilon)$.
Let $\sigma_{i} := \LL{\infty}{1/ \map'}{\omega_{i}}\in (0,1)$ by assumption \eqref{eq:expanding}.
Since   $\map_{i}^{-1} \ball{ \epsilon}{x} \subseteq \ball{\sigma_{i}\epsilon}{\map_{i}^{-1}x}$ we have that 
\begin{equation}\begin{split}\label{eq:back}
\osc{  h\circ \map_{i}^{-1}}{ \ball{\epsilon}{x} \cap \map\omega_{i}} 
&= \osc{ h}{\map_{i}^{-1} \ball{\epsilon}{x} \cap \omega_{i}}\\
&\leq  \osc{   h}{ \ball{\sigma_{i} \epsilon}{y_{i}} \cap \omega_{i}}
\end{split}\end{equation}
where $y_{i} := \map_{i}^{-1}x$.
We change variables in the integral and so
\begin{equation}\begin{split}\label{eq:est1}
 A_{1,\xi,h}(\epsilon) 
&\leq \epsilon^{-\alpha} (2+\tfrac{\delta}{4}) \sum_{i\in\cI} \int_{ \omega_i}\abs{{\xi }}(y_i)
 \cdot \osc{ h}{ \ball{\sigma_{i}\epsilon}{y_i} \cap \omega_{i}} \ dy_i   \\
 &\leq \epsilon^{-\alpha} (2+\tfrac{\delta}{4}) \LL{\infty}{\xi}{\omega_{i}} \int_{ \Oo}
  \osc{ h}{ \ball{\sigma_{i}\epsilon}{y} \cap \Oo} \ dy  \\
  &\leq \sigma_{i}^{\alpha} (2+\tfrac{\delta}{4}) \LL{\infty}{\xi}{\omega_{i}} \abs{h}_{\Aspace}
  \leq  (2+\tfrac{\delta}{4}) \lambda \abs{h}_{\Aspace}.
\end{split}\end{equation}
Now we estimate $ A_{2,\xi,h}(\epsilon)$.
By {\cite[Lemma 2.2]{keller1985generalized}} we have the estimate
\[
 \LL{\infty}{h}{\omega_i} 
 \leq \epsilon_{0}^{-1}\int_{\omega_i} \osc{h}{\ball{\epsilon_0}{x}} \ dx
 + \abs{\omega_i}^{-1} \LL{1}{h}{\omega_i}.
\]
By assumption \eqref{eq:holder} we know that $\osc{\smash{ \tfrac{\xi }{\map'}}}{ \ball{\sigma_{i}\epsilon}{y_i} \cap \omega_{i}} \leq 2 H_\xi \sigma_{i}^\alpha \epsilon^\alpha$ and so, changing variables as per \eqref{eq:back}, we have
 \[
  \int_{\map \omega_i} \osc{ \tfrac{\xi }{\map'}\circ \map_{i}^{-1}}{ \ball{\epsilon}{x} \cap \map\omega_{i}} \ dx   \leq 
  \int_{ \omega_i} \osc{ \tfrac{\xi }{\map'}}{ \ball{\sigma_{i}\epsilon}{y_i} \cap \omega_{i}} \ dy_i
\leq 2 \abs{\omega_i} H_\xi \sigma_{i}^\alpha \epsilon^\alpha.
 \]
 Combining the above two estimates means that
 \begin{equation*}
       A_{2,\xi,h}(\epsilon)
   \leq 
   4(2+\tfrac{\delta}{4}) H_\xi   \sum_{i\in \cI} \sigma_{i}^\alpha \left(    
  \epsilon_0^{-(1-\alpha)} \abs{\omega_i} \epsilon_0^{-\alpha} \int_{\omega_i}  \osc{h}{\ball{\epsilon_0}{x}} \ dx
   +    \LL{1}{h}{\omega_i}.
   \right)
 \end{equation*}
 By the expanding assumption \eqref{eq:expanding} and by \eqref{eq:smallpieces} we know that $\abs{\omega_i} \leq \sigma_{i} \abs{\map\omega_i} \leq 2  \sigma_{i} \epsilon_0 \Gamma$. Using also \eqref{eq:okey} this means that  for all $i\in \cI$
 \[
  \begin{split}
   4(2+\tfrac{\delta}{4}) H_\xi \sigma^\alpha     
  \epsilon_0^{-(1-\alpha)} \abs{\omega_i} 
  &\leq   8 (2+\tfrac{\delta}{4}) H_\xi \LL{\infty}{1/\map'}{\Oo}^{1+\alpha}     
  \epsilon_0^{\alpha} \Gamma 
  \\
  &\leq \tfrac{\delta}{4}\lambda.
  \end{split}
 \]
 Consequently we have shown that
  \begin{equation}\label{eq:est2}
 \begin{split}
       A_{2,\xi,h}(\epsilon)
   &\leq 
   \tfrac{\delta}{4} \lambda \abs{h}_{\Aspace} + 
    4(2+\tfrac{\delta}{4}) H_\xi \LL{\infty}{1/\map'}{\Oo}^\alpha  \LL{1}{h}{\Oo}.
    \end{split}
 \end{equation}
Now we estimate $ A_{3,\xi,h}(\epsilon)$.
Using again {\cite[Lemma 2.2]{keller1985generalized}} we have the estimate
\begin{equation}\label{eq:buda}
 \LL{\infty}{h}{\Oo} \leq
 \epsilon_{0}^{-(1-\alpha)} \abs{h}_{\Aspace} + \abs{\Oo}^{-1} \LL{1}{h}{\Oo}.
\end{equation}
This means that 
\[
  A_{3,\xi,h}(\epsilon)
 \leq
 4 \left(  \abs{h}_{\Aspace} +  \frac{ \epsilon_{0}^{1-\alpha}}{ \abs{\Oo}} \LL{1}{h}{\Oo}\right)
  \sum_{i\in\cI_2}  \LL{\infty}{1/\map' }{\omega_i} \LL{\infty}{\xi }{\omega_{i}} .
\]
By \eqref{eq:remainder} we know that
$  \sum_{i\in\cI_2} \LL{\infty}{1/\map'}{\omega_i}  \LL{\infty}{\xi }{\omega_{i}} \leq \frac{\lambda \delta}{16}$ and so 
\begin{equation}\label{eq:est4}
  A_{3,\xi,h}(\epsilon)
 \leq \tfrac{ \delta}{4}\lambda\abs{h}_{\Aspace} 
 + \left( \frac{ \epsilon_{0}^{1-\alpha} \delta}{4 \abs{\Oo}  } \lambda  \right)
 \LL{1}{h}{\Oo}.
\end{equation}
Now we estimate $ A_{4,\xi,h}(\epsilon)$.
Using again the assumption~\eqref{eq:summing} we may choose a finite set $\cI_{3}\subset\cI_{1}$ such that
\[
\sum_{i\in \cI_{4}}   \LL{\infty}{1/\map'}{\omega_i}  \LL{\infty}{\xi }{\omega_{i}}
\leq \frac{\delta \epsilon_{0}}{4} \lambda,
\]
where $\cI_{4}:= \cI_{1}\setminus\cI_{3}$.
We therefore estimate, using also a change of variables $y_{i} := \map_{i}^{-1}x$,
\[
\begin{split}
A_{4,\xi,h}(\epsilon) 
&\leq \epsilon_0^{-\alpha} \sum_{i\in \cI_{1}} \int_{\map\omega_{i}} \abs{ \tfrac{\xi\cdot h}{\map'}  } \circ\map_{i}^{-1}(x) \ dx\\
&\leq  \epsilon_0^{-\alpha} \sum_{i\in \cI_{3}} \int_{\omega_{i}} \abs{ {\xi\cdot h}}(y_{i})  \ dy_{i}
+  \epsilon_0^{-\alpha} \sum_{i\in \cI_{4}} \LL{\infty}{\frac{\xi}{\map'}}{\omega_{i}} \LL{\infty}{h}{\Oo}
\end{split}
\]
Using \eqref{eq:buda} to estimate $\LL{\infty}{h}{\Oo}$,
this means that for all  $\epsilon\in (0,\epsilon_0)$ we have 
\begin{equation}\label{eq:est3}
 A_{4,\xi,h}(\epsilon)
 \leq \frac{\delta}{4} \lambda \abs{h}_{\Aspace} + \left(\epsilon_0^{-\alpha}\abs{\Oo}^{-1}\sup_{i\in\cI_{3}}\LL{\infty}{\xi}{\omega_{i}}  \right)\LL{1}{h}{\Oo}.
\end{equation}

Summing the estimates of \eqref{eq:est1}, \eqref{eq:est2}, \eqref{eq:est4}, and \eqref{eq:est3} we have shown that 
\[
 \abs{\cL_{\xi}h}_{\Aspace} 
\leq (2+\delta) \lambda \abs{h}_{\Aspace} 
+ C_{\delta} \LL{1}{h}{\Oo}
\]
for all $h\in \Aspace$,
where
\[
 C_{\delta}:= 
 4(2+\tfrac{\delta}{4}) H_\xi \LL{\infty}{\tfrac{1}{\map'}}{\omega_{i}}^\alpha
 +
\frac{1}{ \epsilon_0^{\alpha}\abs{\Oo}}\sup_{i\in \cI_{3}}\LL{\infty}{\xi}{\omega_{i}} 
 +
  \frac{\lambda \epsilon_{0}^{1-\alpha} \delta}{4 \abs{\Oo}}.  
\]
This completes the proof of Theorem~\ref{thm:LY}.

\section{Twisted Transfer Operators}\label{sec:proof}
In this section we follow the standard ``twisted transfer operator'' approach to studying flows. We will take steps to allow the transfer operator results of the previous section to be applied to the original problem of the meromorphic extension of the correlation function. Throughout this section we suppose that we are given a suspension semiflow $(\Oo,\map,\tau)$ which satisfies the assumptions of the Main Theorem, in particular satisfying the assumptions \eqref{eq:holderass}, \eqref{eq:expandingass}, and \eqref{eq:sumass}. First we show that a condition named \emph{exponential tails} in \cite{avila2005emt} holds also in this setting.
\begin{lem}\label{lem:exptails}
$\int_{\Oo} e^{\sigma \tau(x)} \ dx <\infty$. 
\end{lem}
\begin{proof}
We estimate $\int_{\Oo} e^{\sigma \tau(x)} \ dx  \leq
\sum_{i\in \cI} \abs{\omega_{i}} e^{\sigma \LL{\infty}{\tau}{\omega_{i}}}$. Since also we have that $\abs{\omega_{i}} \leq \LL{\infty}{1/\map'}{\omega_{i}} \abs{\Oo} $ then the supposition \eqref{eq:sumass} implies the lemma.
\end{proof}

For all $t\geq 0$ let $A_{t} := \{(x,s) \in \Oo_{\tau} : s+t \geq \tau(x)\}$ and $B_{t}:= \Oo_{\tau  } \setminus A_{t}$. Hence we may write
\begin{equation}\label{eq:AandB}
\mu( u \cdot v\circ \flo{t} )  = \mu( u \cdot v\circ \flo{t} \cdot  \indicator{A_{t}})  + \mu( u \cdot v\circ \flo{t} \cdot  \indicator{B_{t}} ).
\end{equation}
Exponential decay for the second term is  simple to estimate. 
\begin{lem}\label{lem:simplepart}
Exists $C<\infty$ such that 
$\abs{    \mu( u \cdot v\circ \flo{t} \cdot  \indicator{B_{t}} ) } \leq  C \abs{u}_{\infty} \abs{v}_{\infty}  e^{-\sigma t}$ for all $u, v: \Oo_{\tau} \to \bC$ bounded and $t\geq 0$.
\end{lem}
\begin{proof}
Since $\mu$ is given by a formula in terms of the measure $\nu$ which  is absolutely continuous with respect to Lebesgue there exists  $C<\infty$ such that, letting $D_{t} := \{x\in \Oo: \tau(x)-t>0\}$, we have 
\begin{equation}\label{eq:janice}
\abs{    \mu( u \cdot v\circ \flo{t} \cdot  \indicator{B_{t}} ) } \leq C \abs{u}_{\infty} \abs{v}_{\infty}   \int_{\Oo} ( \tau(x)-t) \cdot \indicator{D_{t}}(x)  \ dx
\end{equation} 
for all $t\geq 0$.
For $y\in \bR$ we define $k(y)$ equal to $y$ if $y\geq 0$ and equal to $0$ otherwise. 
This definition  means that $   ( \tau(x)-t) \cdot \indicator{D_{t}}(x) \leq k( \tau(x)-t)  $.
Since $\ln y \leq y$ for all $y>0$ it follows that $\ln ( \sigma y) = \ln \sigma + \ln y \leq \sigma y$ and so $y\leq \sigma^{-1} e^{\sigma y}$ for all  $y>0$. The case $y\leq 0$ is simple and so we have shown that $k(y) \leq \sigma^{-1} e^{\sigma y}$ for all  $y\in \bR$.
This means that 
\[
(\tau(x)-t) \cdot \indicator{D_{t}} (x) \leq  \sigma^{-1} e^{\sigma (\tau(x)-t)}, \quad \text{ for all $x\in \Oo$}.
\]
 We conclude using the above with \eqref{eq:janice} since $\int e^{\sigma \tau (x) }\ dx <\infty$ by Lemma~\ref{lem:exptails}.
\end{proof}

In order to proceed we must estimate the other term in \eqref{eq:AandB} and so it is convenient to define
\begin{equation}\label{eq:defrho}
\rho(t):= \mu( u \cdot v\circ \flo{t} \cdot  \indicator{A_{t}}).
\end{equation}
Note that $\abs{  \mu( u \cdot v\circ \flo{t} \cdot  \indicator{A_{t}})  } \leq \abs{u}_{\infty} \abs{v}_{\infty} $ for all $t\geq 0$.
For all $z\in \bC$ such that $\Re(z)>0$ we consider the Laplace transform of the above function
\begin{equation}\label{eq:defLaplace}
\hat\rho(z):= \int_{0}^{\infty} e^{-zt} \rho(t) \ dt.
\end{equation}
Additionally for any $u:\Oo_{\tau} \to \bC$ and $z\in \bC$ let
\begin{equation}\label{eq:defflattening}
\hat u_{z}(x):= \int_{0}^{\infty} e^{-zs} u(x,s) \ ds
\end{equation}
for all $x\in \Oo$. Furthermore for all $n\in \bN$ let $\tau_{n} := \sum_{k=0}^{n-1} \tau\circ \map^{k}$.  
Since the invariant measure $\nu$  is absolutely continuous with respect to Lebesgue there exists a density $h_{0} \in \Lspace{1}{\Oo}$ such that $
\mu (\eta) = \int_{\Oo} \int_{0}^{\tau(x)} \eta(x,s)\ ds \  h_{0}(x) \ dx
$
for all bounded $\eta:\Oo_{\tau} \to \bC$.
As in \cite{Po,Pol99,baladi2005edc,avila2005emt} we have the following representation of the Laplace transform in terms of an infinite sum.

\begin{lem}\label{lem:sum}
For all $z\in \bC$ such that $\Re(z)>0$ and all $\abs{u}_{\infty}<\infty$, $\abs{v}_{\infty}<\infty$
\[
\hat\rho(z) = \sum_{n=1}^{\infty} \int_{\Oo} (h_{0}\cdot \hat u_{-z}\cdot e^{-z\tau_{n}}\cdot \hat v_{z} \circ \map^{n})(x)   \ dx.
\]
\end{lem}
\begin{proof}
Recall that $h_{0}\in \Lspace{1}{\Oo}$ is the density of the $\map$-invariant measure $\nu$.
For all $\Re(z)>0$ 
\[
\begin{split}
\hat\rho(z) &= \int_{0}^{\infty}\int_{\Oo}\int_{0}^{\tau(x)} e^{-zt} u(x,s) v\circ \flo{t}(x,s) \indicator{A_{t}}(x,s) h_{0}(x) \ ds \ dx \ dt\\
&= \sum_{n=1}^{\infty} \int_{\Oo}\int_{0}^{\tau(x)}  \int_{\tau_{n}(x)-s}^{\tau_{n+1}(x) -s} e^{-zt} u(x,s) v\circ \flo{t}(x,s)  h_{0}(x) \ dt \ ds \ dx.
\end{split}
\]
We change variables letting $t' = t - \tau_{n}(x) + s$ and note that when $t\in [\tau_{n}(x)-s, \tau_{n+1}(x)-s]$ then $\flo{t}(x,s) = (\map^{n}x,  t - \tau_{n}(x) + s)$.  This means that
\begin{multline*}
\hat\rho(z)
=  \sum_{n=1}^{\infty} \int_{\Oo} e^{-z\tau_{n}(x)} \left(  \int_{0}^{\tau(x)} e^{zs} u(x,s) \ ds\right) \\ \times \left( \int_{0}^{\tau(\map^{n}x)} e^{-zt'} v(\map^{n}x, t') \ dt'  \right) h_{0}(x) \ dx.
\end{multline*}
Recalling the definition  \eqref{eq:defflattening} for $\hat u_{-z}$ and $\hat v_{z}$ we conclude.
\end{proof}

We now relate 
the sum given by Lemma~\ref{lem:sum} to the \emph{twisted transfer operators}.
For all $z\in \bC$  such that $\Re(z)\in[-\sigma,0]$ let $\xi_{z}: \Oo \to \bC$ be defined as
\begin{equation}\label{eq:defweight}
 \xi_{z}:= {e^{-z\tau} }.
\end{equation}
We consider the map $f:\Oo\to\Oo$ with the weighting $\xi_{z}$.
It is immediate that the assumptions imposed on the semiflow imply that the pair $f$ and $\xi_{z}$ satisfy the assumptions of Theorem~\ref{thm:qc}. Consequently the transfer operator   $\cL_{z}:\Aspace\to\Aspace$ (for convenience we now write $\cL_{z}$ for $\cL_{\xi_{z}}$) and which is given by the formula
\[
\cL_{z}h(x):= \sum_{i\in\cI} \left(\frac{e^{-z\tau} \cdot h}{\map'}\right)\circ \map_{i}^{-1}(x)\cdot \indicator{\map\omega_{i}}(x).
\]
has essential spectral radius strictly less than $1$.
Let $ \cB(\Aspace,\Aspace)$ denote the space of bounded linear operators mapping $\Aspace$ to $\Aspace$.
\begin{lem}\label{lem:meromorphic}
The operator valued function $z\mapsto (\Id - \cP_{z})^{-1} \in \cB(\Aspace,\Aspace)$ is meromorphic on the set $\{z\in \bC: \Re(z)\in[-\sigma,0] \}$.
\end{lem}
\begin{proof}
We know that $\cP_{z} \in \cB(\Aspace,\Aspace)$ has essential spectral radius less than $1$ for all $  \Re(z)\in[-\sigma,0]  $ and so is of the form $\cP_{z} = \cK_{z} + \cA_{z}$ where $\cK_{z}$ is compact,  the spectral radius of $\cA_{z}$ is strictly less than $1$ and $\cK_{z} \cA_{z} = 0$. Furthermore both $z\mapsto \cK_{z} \in  \cB(\Aspace,\Aspace)$ and  $z\mapsto \cA_{z} \in  \cB(\Aspace,\Aspace)$ are holomorphic operator-valued  functions.
Note that 
\[
(\Id - \cP_{z}) = (\Id - \cK_{z}) (\Id - \cA_{z}).
\]
and that $(\Id - \cA_{z})$ is invertible.
By The Analytic Fredholm Theorem $z\mapsto (\Id - \cK_{z})^{-1}$ is meromorphic on the set $\{z\in \bC: \Re(z)\in[-\sigma,0]\}$.
\end{proof}

\begin{lem}
The operator valued function 
$
z\mapsto \sum_{n=1}^{\infty} \cP_{z}^{n}\in \cB(\Aspace,\Aspace)
$
 is meromorphic on the set $\{z\in \bC: \Re(z)\in[-\sigma,0]\}$.
\end{lem}
\begin{proof}
We note that $ \sum_{n=1}^{\infty} \cP_{z}^{n} =  (\Id - \cP_{z})^{-1} \cP_{z} $ and apply Lemma~\ref{lem:meromorphic}.
\end{proof}

\begin{proof}[Proof of  \citemain]
By Lemma~\ref{lem:simplepart} it suffices to know that $\hat\rho$ admits the relevant meromorphic extension.
Since, as usual for transfer operators, we have that
\[
\int_{\Oo} \cP_{z}^{n}h_{1}(x) \cdot h_{2}(x) \ dx =
\int_{\Oo} h_{1}(x) \cdot h_{2} \circ e^{-z\tau_{n}(x)} \circ \map^{n}(x) \ dx  
\]
the formula for $\hat\rho(z) $  given by Lemma~\ref{lem:sum} means that
\[
\hat \rho(z) = \sum_{n=1}^{\infty} \int_{\Oo} \cL_{z}^{n}(h_{0}\hat u_{-z})(x) \cdot \hat v_{z}(x) \ dx.
\]
This equality was shown to hold for all $\Re(z)>0$. But since the right hand side is meromorphic on the set $\{z\in \bC : \Re(z)\in[-\sigma,0]\}$ we have shown that the left hand side admits such an extension. 
\end{proof}

\end{document}